\documentclass{amsart}

\usepackage{amsmath}
\usepackage{amssymb}
\usepackage{amsthm}
\usepackage{listings}
\usepackage{float}

\usepackage{breqn}

\usepackage{url}
\usepackage{hyperref}
\usepackage{cleveref}

\usepackage{breakurl}
\usepackage{tikz,tabularx}

\usepackage{amssymb,amsmath,amsthm,units,esint}

\usepackage{color,graphicx} 			
\usepackage[small]{caption}
\usepackage{cases}

\usepackage{caption} 
\allowdisplaybreaks

\newtheorem{theorem}{Theorem}

\theoremstyle{remark}
\newtheorem{remark}{Remark}
\newtheorem*{remark*}{Remark}

\newtheorem{corollary}{Corollary}

\newcommand{\sech}{\ensuremath{\operatorname{sech}}}

\renewcommand{\Im}{\mathop{\rm Im}\nolimits}
\renewcommand{\Re}{\mathop{\rm Re}\nolimits}
\newcommand{\res}{\mathop{\rm res}\limits}

\newcommand{\ch}{\cosh}
\newcommand{\sh}{\sinh}
\renewcommand{\th}{\tanh}
\newcommand{\cth}{\coth}
\renewcommand{\ln}{\log}
\newcommand{\arctg}{\arctan}
\newcommand{\atan}{\arctan}

\newcommand{\be}{\begin{equation}}
\newcommand{\ee}{\end{equation}} 
\makeatletter
\newcommand{\specialnumber}[1]{
\def\tagform@##1{\maketag@@@{(\ignorespaces##1\unskip\@@italiccorr#1)}}}
\makeatother

\begin{document}

\title{{Computing Stieltjes constants using complex integration}}

\author{Fredrik Johansson}
\address{LFANT -- INRIA -- IMB, Bordeaux, France}
\email{fredrik.johansson@gmail.com}
\date{}

\author{Iaroslav V.~Blagouchine}
\address{SeaTech, University of Toulon, France}
\email{iaroslav.blagouchine@univ-tln.fr}


\subjclass[2010]{Primary 11M35, 65D20; Secondary 65G20}

\date{XXX and, in revised form, XXX}

\keywords{Stieltjes constants, Hurwitz zeta function, Riemann zeta function, integral representation, complex integration, 
numerical integration, complexity, arbitrary-precision arithmetic, rigorous error bounds.}

\begin{abstract}
The generalized Stieltjes constants $\gamma_n(v)$ are, up to a simple scaling factor,
the Laurent series coefficients of the Hurwitz zeta function $\zeta(s,v)$ about its unique pole $s = 1$.
In this work, we devise an efficient algorithm to compute these constants
to arbitrary precision with rigorous error bounds,
for the first time achieving this with low complexity with respect to the order~$n$.
Our computations are based on an integral representation
with a hyperbolic kernel that decays exponentially fast.
The algorithm consists of locating an approximate steepest descent contour
and then evaluating the integral numerically in ball arithmetic
using the Petras algorithm with a Taylor expansion for bounds
near the saddle point. 
An implementation is provided in the Arb library.
We can, for example, compute $\gamma_n(1)$ to 1000 digits
in a minute for any $n$ up to $n=10^{100}$. We also provide other interesting 
integral representations for $\gamma_n(v)$, 
$\zeta(s)$, $\zeta(s,v)$, some polygamma functions and the Lerch transcendent.
\end{abstract}

\maketitle

\section{Introduction}

The Hurwitz zeta function $\zeta(s,v) = \sum_{k=0}^{\infty} (k+v)^{-s}$
is defined for all complex $v \ne 0, -1, -2, \ldots$ and
by analytic continuation for all complex~$s$
except for the point $s = 1$, at which it has a simple pole. The Laurent series in a neighborhood of this unique pole
is usually written as
\begin{equation}
\label{eq:defstieltjes}
\zeta(s,v)=\frac{1}{s-1}+\sum_{n=0}^\infty \frac{(-1)^n}{n!} \gamma_n(v) (s-1)^n\,, \qquad s\in\mathbb{C}\setminus\{1\}.
\end{equation}
The coefficients $\gamma_n(v)$ are known as the generalized Stieltjes constants.
The ordinary Stieltjes constants $\gamma_n = \gamma_n(1)$, appearing in the analogous 
expansion of the Riemann zeta function $\zeta(s) = \zeta(s,1)$, are also known as the generalized Euler constants and include
the Euler-Mascheroni constant
$\gamma_0 = \gamma = 0.5772156649\ldots$ as a special case.\footnotemark

This work presents an original method to compute $\gamma_n(v)$ rigorously
to arbitrary precision, with the property of remaining fast for arbitrarily large $n$.
Such an algorithm has never been published (even in the case
of $v = 1$), despite an extensive
literature dedicated to the Stieltjes constants.
At the heart of the method is Theorem~\ref{thm:integrals},
given below in Section~\ref{sect:integrals}, which provides
computationally viable
integral representations for $\zeta(s,v)$ and $\gamma_n(v)$.
In particular, for $n \in \mathbb{N}_{0}$ and $\Re(v) > \frac12$,
\begin{equation}
\label{eq:coshintegral}
\gamma_n(v) = -\frac{\pi}{\,2(n+1)\,}\!\int_{-\infty}^{+\infty} 
\!\frac{\,\ln^{n+1}\!\big(v-\frac12 + ix\big)\,}{\ch^2 \! \pi x}\,dx,
\end{equation}
which extends representation (5) from \cite{blagouchine2015theorem} to the generalized Stieltjes constants.
The above expression is similar to the Hermite formula
\begin{equation}
\label{eq:hermite}
\gamma_n(v) = \left(\frac{1}{2v}\!-\!\frac{\log v}{n+1}\right) \log^n \! v - i\!
\int_0^{\infty} \!\!\!\frac{dx}{e^{2\pi x}-1} 
\!\!\left\{\frac{\log^n(v\!-\!ix)}{v-ix}\!-\!\frac{\log^n(v\!+\!ix)}{v+ix} \right\},
\end{equation}
see e.g.~\cite[Eq.~13]{blagouchine2015theorem},
but more convenient to use for computations since the integrand with the hyperbolic kernel $\sech^2 \! \pi x$ 
does not possess a removable singularity at \mbox{$x = 0.$} Additionally, Section~\ref{sect:integrals} provides some other 
integral representations for $\zeta(s,v)$ and $\gamma_n(v)$, some of which may also be suitable for computations
(see, in particular, Corollary \ref{col1} and Remark \ref{y78v98j2}).

\footnotetext{Its generalized analog $\gamma_0(v)$
includes the digamma function $\Psi(v)$, namely $\gamma_0(v)=-\Psi(v)$, see e.g.~\cite[Eq.~(14)]{blagouchine2015theorem}.}

Section~\ref{sect:computation}
describes a robust numerical integration strategy for computing $\gamma_n(v)$.
A crucial step is to determine an approximate steepest descent contour
that avoids catastrophic oscillation for large $n$.
This is combined with validated integration to ensure
that the computation is accurate (indeed, yielding proven error bounds).
An open source implementation is available in the Arb library~\cite{Johansson2017arb}.
Section~\ref{sect:benchmark} contains benchmark results.

\subsection{Background}

The numbers $\gamma_n$ with $n \le 8$ were first computed to
nine decimal places by Jensen in 1887. Many authors have
followed up on this work using an array of techniques.
Fundamentally, any method to compute~$\zeta(s)$ or $\zeta(s,v)$ can be adapted to compute $\gamma_n$ or $\gamma_n(v)$
respectively by taking derivatives.
For example, as discussed by Gram \cite{gram1895}, Liang and Todd~\cite{liang1972stieltjes},
Jensen's calculations of $\gamma_n$ were based on the limit representation
\be\label{34f34ght4g}
\zeta(s) = \lim_{N \to \infty} \!\left\{\sum_{k=1}^N \frac{1}{k^s} - \frac{N^{1-s}}{1-s}\right\}\,,\qquad \Re(s)>0\,,
\ee
which follows from the Euler--Maclaurin summation formula, so that
$$\gamma_n = \lim_{N \rightarrow \infty}\!{\left\{\sum_{k=1}^N  \frac{\log^n k}{k} - \frac{\log^{n+1} N}{n+1}\right\}}
\,,\qquad n\in\mathbb{N}_0\,,$$
while Gram expressed $\gamma_n$ in terms of derivatives
of the Riemann $\xi$ function which he evaluated using integer zeta values.
Liang and Todd proposed computing~$\gamma_n$ either
via the Euler-Maclaurin summation formula for $\zeta(s)$,
or, as an alternative, via the application of Euler's series transformation
to the alternating zeta function.
Bohman and Fr\"{oberg} \cite{bohman1988stieltjes} later refined the limit formula technique.

Formula \eqref{eq:hermite} is a differentiated form of
Hermite's integral representation for $\zeta(s,v)$, which can
be interpreted as the Abel-Plana summation formula applied to
the series for $\zeta(s,v)$.
As discussed by Blagouchine~\cite{blagouchine2015theorem},
the formula \eqref{eq:hermite} has been rediscovered several times
in various forms (the $v = 1$ case should be credited to Jensen and Franel and dates back to the end of the XIXth century).
Ainsworth and Howell~\cite{ainsworth1985} rediscovered
the $v = 1$ case of \eqref{eq:hermite}
and were able to compute $\gamma_n$ up to $n = 2000$
using Gaussian quadrature.

Keiper~\cite{keiper1992power} proposed an algorithm
based on the approximate functional equation
for computing the Riemann $\xi$ function, which upon
differentiation yields derivatives of $\xi(s)$ as integrals
involving Jacobi theta functions. The Stieltjes constants
are then recovered by a power series transformation.

Kreminski~\cite{kreminski2003newton} used a version of the limit formula
combined with Newton-Cotes quadrature to estimate the resulting sums,
and computed accurate values of $\gamma_n(v)$ up to $n = 3200$
for $v = 1$ and up to $n = 600$ for various rational $v$.
More recently, Johansson~\cite{Johansson2014hurwitz}
combined the Euler-Maclaurin formula with fast power series arithmetic
for computing $\gamma_n(v)$,
proved rigorous error bounds for this method,
and performed the most extensive computation of Stieltjes constants to date resulting in
10000-digit values of~$\gamma_n$ for all $n \le 10^5$.

Even more recently, Adell and Lekuona~\cite{adell2017fast}
have used probability densities for
binomial processes to obtain new rapidly convergent series for $\gamma_n$
in terms of Bernoulli numbers.

The drawback of the previous methods
is that the complexity to compute~$\gamma_n$
is at least linear in~$n$.
In most cases, the complexity is actually at least quadratic in~$n$ since
the formulas tend to have a high degree of cancellation necessitating
use of $\Omega(n)$-digit arithmetic.
For the same reason, the space complexity is also usually quadratic in~$n$,
at least in the most efficient forms of the algorithms.
For numerical integration of~\eqref{eq:hermite}, the difficulty
for large $n$ lies in the oscillation of the integrand
which leads to slow
convergence and catastrophic cancellation.\footnote{The formula \eqref{eq:hermite}
has also been used by Johansson for a numerical implementation of Stieltjes constants
in the mpmath library~\cite{johansson2017mpmath},
but the algorithm as implemented in mpmath
loses accuracy for large $n$
(for example, $\gamma_{10^4} \approx -2.21 \cdot 10^{6883}$
but mpmath~1.0 computes $-1.25 \cdot 10^{6800}$).}

The fast Euler-Maclaurin method~\cite{Johansson2014hurwitz}
does allow computing $\gamma_0, \ldots, \gamma_n$
simultaneously to a precision of $p$ bits in $n^{2+o(1)}$ time if $p = \Theta(n)$,
which is quasi-optimal. However, this is not ideal if we only need $p = O(1)$
or a single $\gamma_n$.\footnote{It is of course also interesting
to consider the complexity of computing
a single $\gamma_n$ value to variable accuracy $p$.
For $n \ge 1$, the complexity is $\Omega(p^2)$
with all known methods (although the fast Euler-Maclaurin
method amortizes this to $p^{1+o(1)}$ per coefficient when computing $n = \Theta(p)$
values simultaneously).
The exception is $\gamma_0$ which
can be computed in time $p^{1+o(1)}$ by exploiting its role as a
hypergeometric connection constant \cite{brent1980some}.}

This leads to the question of whether we can compute $\gamma_n$
quickly for any~$n$; ideally, in time depending only polynomially on $\log(n)$.
If we assume that the accuracy goal $p$ is fixed, then any
asymptotic formula $\gamma_n \sim G(n)$ where $G(n)$ is an easily
computed function should do the job.

Various asymptotic estimates and bounds for the Stieltjes constants have
been published, going back at least to Briggs~\cite{briggs1955some}
and Berndt~\cite{berndt1972hurwitz},\footnote{For the more complete history, see \cite[Sect.~3.4]{blagouchine2016}.}
but the explicit computations by Kreminski and others showed
that these estimates were far from precise.

A breakthrough came in 1984, when Matsuoka succeeded in obtaining the first--order asymptotics for
the Stieltjes constants \cite{matsuoka1985}, \cite[p.~3]{eddinthese2013}.
Four years later he derived the complete asymptotic expansion 
\be\label{matsuokabest}
\begin{array}{rl}
\displaystyle\gamma_n\,\sim & \displaystyle\frac{n!\,e^{g(b)}}{\pi} \!\sum_{k=0}^N \frac{|h_{2k}| \,2^{k+\frac12}\,
\Gamma(k+\frac12)}{\Big(g''(b)^2+f''(b)^2\Big)^{\!\frac12k +\frac14}}\,\times \\[9mm]
&\displaystyle \quad
\times \cos\!\left[f(b)-\big(k+\tfrac12\big)\arctan\frac{f''(b)}{g''(b)}+\arctan\frac{v_{2k}}{u_{2k}} \right],\quad N=0,1,2,\ldots
\end{array}
\ee
where $h_{k}$, $v_{k}$ and $u_{k}$ are the sequences of numbers defined by 
\begin{eqnarray}
&& \sum_{k=0}^\infty h_k (y-b)^k \,=\, \exp\!\Big[\phi(a+iy) - \phi(a+ib)+\tfrac12\phi''(a+ib) (y-b)^2\Big]\,, \notag\\[1mm]
&& u_{k}\equiv\Re(h_k)		\,, \qquad v_{k}\equiv\Im(h_k)	\,,		 \notag
\end{eqnarray}
and $f(b)$ and $g(b)$ are the functions defined as
\begin{eqnarray}
&& g(y)\equiv\Re\phi(a+iy)		\,, \qquad f(y)\equiv\Im\phi(a+iy)	\,, \notag\\[2mm]
&& \phi(z)\,=\,-(n+1)\log z -z\log2\pi i + \log\Gamma(z) 	\,. \notag
\end{eqnarray}
The pair $a=\Re z$, $b=\Im z$, is the unique solution of the equation
\be\label{93udj2oihnd}
\frac{\,d\phi(z)\,}{dz} \,=\, -\frac{\,n+1\,}{z} - \log2\pi i   + \Psi(z)\,=\,0\,,
\ee
satisfying $0<\Im z<\Re z$ and $\sqrt{n}<\Re z<n$, 
where $\Gamma(z)$ and $\Psi(z)$ are the gamma and digamma 
functions respectively, see \cite[pp.~49--50]{matsuoka1989}.\footnote{Matsuoka's Lemma 1
may be written in our form \eqref{93udj2oihnd} if we notice that equations (2) and (3) \cite[p.~49]{matsuoka1989} actually represent 
one single equation in which real and imaginary parts were written separately with $z=x+iy$, and then recall that $\,\overline{z}/|z|^2=1/z$,
where $\overline{z}$ is the \mbox{complex conjugate of $z.$}} 
The Matsuoka expansion~\eqref{matsuokabest} accurately predicts the behavior of $\gamma_n$, but is very cumbersome to use.
In 2011 Knessl and Coffey~\cite{knessl2011effective} presented a simpler asymptotic
formula\footnote{If we  put $N=0$
in Matsuoka's expansion \eqref{matsuokabest}, we retrieve, after some calculations
and several approximations, the same result as Knessl and Coffey \eqref{eq:knessl}.}
\begin{equation}
\label{eq:knessl}
\gamma_n \sim \frac{B}{\sqrt{n}} e^{nA} \cos(an+b) 
\end{equation}
in terms of the slowly varying functions
$$A = \frac{1}{2} \log(\alpha^2+\beta^2) - \frac{\alpha}{\alpha^2+\beta^2}, \quad B = \frac{2 \sqrt{2\pi(\alpha^2+\beta^2)}}{\sqrt[4]{(\alpha+1)^2+\beta^2}},$$
$$a = \atan\frac{\beta}{\alpha}+ \frac{\beta}{\alpha^2+\beta^2}, \quad b = \atan\frac{\beta}{\alpha} - \frac{1}{2} \atan\frac{\beta}{\alpha+1},$$
where $\beta$ is the unique solution of 
\be\notag
2 \pi \exp(\beta \tan \beta) = \frac{n \cos \beta}{\beta}\,,\qquad \text{with }\quad 0 < \beta < \tfrac12\pi, \quad \alpha = \beta \tan\beta\,.
\ee
In \eqref{eq:knessl}, the ``$\sim$'' symbol signifies asymptotic
equality as long as the cosine factor is bounded away from zero.
The factor $B e^{nA} / \sqrt{n}$ captures the
overall growth rate of $\gamma_n$ while the cosine factor
explains the local oscillations (and semi-regular sign changes).

More recently, Fekih-Ahmed \cite{fekih2014new} has given an alternative
asymptotic formula with similar accuracy to \eqref{eq:knessl}.
Paris~\cite{paris2015asymptotic} has also generalized \eqref{eq:knessl}
to $\gamma_n(v)$ and extended the result to
an asymptotic series with higher order correction terms,
permitting the determination of several digits
for moderately large $n$.

The Matsuoka, Knessl-Coffey, Fekih-Ahmed and Paris formulas were
obtained using the standard asymptotic technique of applying saddle point
analysis to a suitable contour integral.
From a computational point of view, these formulas still have three
drawbacks. First, being asymptotic in nature,
they only provide a fixed level of accuracy
for a fixed~$n$, so a different method must be used
for small $n$ and high precision~$p$.
Second, the terms in Matsuoka's expansion and the high-order terms in Paris's expansion 
are quite complicated to compute.
Third, explicit error bounds are not currently available.

A~natural approach to construct an algorithm with the desired properties
is to take a similar integral representation and
perform numerical integration instead of
developing an asymptotic expansion symbolically.
The integral representations behind the previous asymptotic formulas
do not appear to be convenient for this purpose, since they
involve nonsmooth functions (periodic Bernoulli polynomials)
or require a summation over several integrals.
We therefore use integrals with exponentially decreasing kernels, as
in the previous computational work by Ainsworth and Howell~\cite{ainsworth1985},
but with the addition of saddle point analysis (which is necessary
to handle large $n$) and a rigorous treatment of error bounds.

\section{Integral representations}

We obtain the following formulas
in terms of elementary integrands that are
rapidly decaying and analytic on the path of integration.
Although restricted to $\Re(v) > \frac{1}{2}$, they
permit computation on the whole $(s,v)$ and $(n,v)$ domains through application
of the recurrence relations
\begin{equation}
\label{eq:recurrence}
\zeta(s,v) = \zeta(s,v+1) + \frac{1}{v^s}, \qquad \gamma_n(v) = \gamma_n(v+1) + \frac{\log^n \! v}{v}.
\end{equation}

\label{sect:integrals}

\begin{theorem}
\label{thm:integrals}
The Hurwitz zeta function $ \zeta(s,v)$
and the generalized Stieltjes constants $\gamma_n(v)$ may be represented by the following
integrals
\begin{eqnarray}
\zeta(s,v)  && =\, \frac{\pi}{\,2(s-1)\,}\!\int_{-\infty}^{+\infty} \!
\!\frac{\,\big(v-\frac12\pm ix\big)^{1-s}}{\ch^2 \! \pi x}\,dx \, \label{9823ydhdhs7}\\[1mm]
&& =\, \frac{\pi}{\,2(s-1)\,}\!\int_{0}^{\infty} \!
\frac{\big(v-\tfrac12 - ix\big)^{1-s} + \big(v-\tfrac12 + ix\big)^{1-s}} {\,\ch^2 \! \pi x\,}\, {dx} \,\\[1mm]
&& =\, \frac{\pi}{\,2(s-1)\,}\!\! \int_{0}^{\infty}\!\!
\! \frac{\, \cos\!\Big[(s-1)\arctg\tfrac{2x}{2v-1}\Big] \,}{\,\big(v^2 - v + \tfrac14+ x^2\big)^{\frac12(s-1)} \ch^2 \! \pi x \,}\,dx \label{erokvrekov}
\end{eqnarray}
and
\begin{eqnarray}\label{9834df2b3e}
\gamma_n(v) && =\, -\frac{\pi}{\,2(n+1)\,}\!\int_{-\infty}^{+\infty} 
\!\frac{\,\ln^{n+1}\!\big(v-\frac12\pm ix\big)\,}{\ch^2 \! \pi x}\,dx \,\\[1mm]
&& =\, -\frac{\pi}{\,2(n+1)\,}\!\int_{0}^{\infty} \! \label{9834df2b3e2}
\frac{\,\ln^{n+1}\!\big(v-\tfrac12 - ix\big) + \ln^{n+1}\!\big(v-\tfrac12 + ix\big) \,} {\,\ch^2 \! \pi x\,} \, dx
\end{eqnarray}
respectively. All formulas~hold for complex $v$ and $s$ such that
$\Re(v) > \frac12$ and $s\neq1$.\footnote{In these formulas ``$\pm$'' signifies that either sign can be taken.
Throughout this paper
when several ``$\pm$'' or ``$\mp$'' are encountered in the same formula, it signifies that either the upper signs 
are used everywhere or the lower signs are used everywhere (but not the mix of them).}
\end{theorem}

In order to prove the above formulas, we will use the contour integration method.\footnote{Note that since many formulas 
with the kernels decaying exponentially fast
were already obtained in the past by Legendre, Poisson, Binet, Malmsten, Jensen, Hermite, Lindel\"of and many others (see e.g. a formula
for the digamma function on p.~541 \cite{blagouchine2015theorem}, or \cite{blagouchine2014malmsten} or \cite{lindelof1905}),
it is possible that formulas similar 
or equivalent to those we derive in this section might appear in earlier sources of which we are not aware. 
In particular, after the publication of the second draft version of this work, 
we learnt that a formula equivalent to our \eqref{erokvrekov} appears in two books by
Srivastava and Choi, \cite[p.~92, Eq.~(23)]{srivastava_03} and \cite[p.~160, Eq.~(23)]{srivastava_04} respectively. In both sources
it appears without proof and without references to other sources.}

\begin{proof} 
Consider the following line integral taken along a contour $C$ consisting of the interval $[-R,+R]$ on the real axis
and a semicircle of the radius $R$ in the upper half-plane, denoted $C_R$,
\be\label{0923un}
\ointctrclockwise_{C} \frac{\,\big(a - iz\big)^{1-s} \,}{\,\ch^2 \! \pi z \,}\,dz\,=
\int_{-R}^{+R} \! \frac{\,\big(a - ix\big)^{1-s} \,}{\,\ch^2 \! \pi x \,}\,dx \, +  \int_{C_R} \!
\frac{\,(a - iz)^{1-s} \,}{\,\ch^2 \! \pi z \,}\,dz \,.
\ee
On the contour $C_R$ the last integral may be bounded as follows:
\begin{eqnarray}\notag
&& \left|\, \int_{C_R} \!
\frac{\,(a- iz)^{1-s} \,}{\,\ch^2 \! \pi z \,}\,dz \, \right|   \,=\, R
\left|\, \int_0^{\pi} \!
\frac{\,\big(a - iRe^{i\varphi}\big)^{1-s} e^{i\varphi}\,}{\,\ch^2 \! \big(\pi Re^{i\varphi}\big) \,}\,d\varphi \, \right|  \,\le\\[2mm]
&& \qquad \qquad\notag
\,\le \,R \!\! \max_{\varphi\in[0,\pi]} \!
\left| \big(a - iRe^{i\varphi}\big)^{1-s}  \right| \cdot I_R  \\[2mm]
&& \qquad \qquad \le\,  R \!\! \max_{\varphi\in[0,\pi]} \!
\Big[|a|^2 +2R \,(a_x\sin\varphi-a_y\cos\varphi) +R^2\Big]^{\frac12\Re{(1-s)}} e^{\pi|\Im(1-s)|}  \, I_R\label{98ydfn2k}
\end{eqnarray}
where we denoted $a_x\equiv\Re(a)$, $a_y\equiv\Im(a)$ and
\be\notag
I_R\,\equiv\int_0^{\pi} \!
\frac{\,d\varphi \,}{\,\big|\ch\! \big(\pi Re^{i\varphi}\big) \big|^2\,}\,,\qquad R>0\,,
\ee
for the purpose of brevity. It can be shown that as $R$ tends to infinity and remains integer the integral $I_R$ tends to zero 
as $1/R$. For this aim, we first remark that 
\be\notag
\frac{\,1\,}{\,\big|\ch\! \big(\pi Re^{i\varphi}\big) \big|^2\,}\,=\,\frac{2}{\,\ch(2\pi R\cos\varphi) + \cos(2\pi R\sin\varphi)\,}\,.
\ee
Since $R$ and $\varphi$ are both real, $\ch(2\pi R\cos\varphi) >1 $ 
except for the case when $\cos\varphi=0$. Hence
\be\notag
\ch(2\pi R\cos\varphi) + \cos(2\pi R\sin\varphi)>0\,,
\ee
except perhaps at $\varphi=\frac{1}{2}\pi$. But
at the latter point, since $R$ is integer,
\be\notag
\left. \ch(2\pi R\cos\varphi) + \cos(2\pi R\sin\varphi) \vphantom{\int}\right|_{\varphi=\frac{1}{2}\pi} \!\!\!\! =\,1+\cos(2\pi R)\,=\,2\,.
\ee Therefore $\big|\ch \! \big(\pi Re^{i\varphi}\big)\big|^{-2}$ remains always bounded for integer $R$
(see also Fig.~\ref{c3f38nc389y}),
\begin{figure}[!t]    
\centering 
\includegraphics[angle=-90,width=0.9\textwidth]{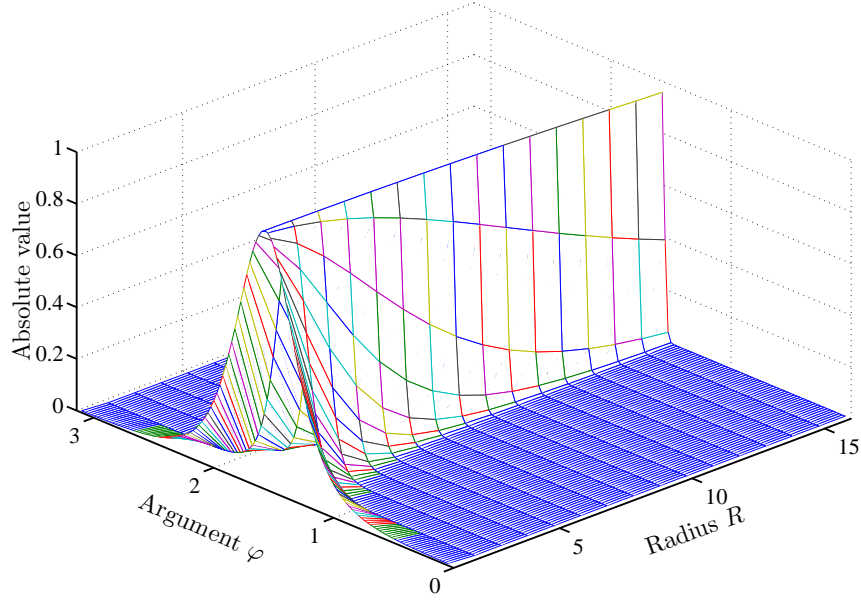} 
\caption{3D-plot of $\,\left|\ch \! \left(\pi R e^{i\varphi}\right)\right|^{-2}\,$
for $R\in[1,16]$ and $\varphi\in[0,\pi]$ clearly displays
the boundness of the latter ($R$ integer). Note also that at large $R$ the contribution 
of the point $\varphi=\frac{1}{2}\pi$ to the integral $I_R$ becomes
infinitely small (its height equals one, while the width tends to zero).} 
\label{c3f38nc389y} 
\end{figure}
and when $R\to\infty$ we have 
\be
\label{uyc429h4}
\frac{\,1\,}{\,\big|\ch\! \big(\pi Re^{i\varphi}\big) \big|^2\,} \, = \,
\begin{cases}
\,O\big(e^{-2\pi R\cos\varphi}\big ) \,, 	& \varphi\in\big[0,\frac{1}{2}\pi \big] \,,\\[4mm]
\, O\big(e^{+2\pi R\cos\varphi}\big ) \,, 	& \varphi\in\big[\frac{1}{2}\pi, \pi\big]\,.
\end{cases}
\ee 
Thus, accounting for the symmetry of $\,\big|\ch \! \big(\pi Re^{i\varphi}\big)\big|^{-2}\,$ about $\varphi=\frac{1}{2}\pi$,
we deduce that 
\begin{eqnarray}
\notag
I_R && =\int_0^\pi \!
\frac{\,2 \, d\varphi \,}{\,\ch(2\pi R\cos\varphi) + \cos(2\pi R\sin\varphi)\,}\\[2mm]
&& \notag
=\int_0^{\frac{\pi}{2}} \!
\frac{\,4\, d\varphi \,}{\,\ch(2\pi R\cos\varphi) + \cos(2\pi R\sin\varphi)\,}\\[2mm]
&&  \label{093ujfo3pjf}
= \, O\!\left(\!\int_0^{\frac{\pi}{2}} \! e^{-2\pi R\cos\varphi} \, d\varphi  \right) 
= \, O\!\left(\!\int_0^{\frac{\pi}{2}} \! e^{-2\pi R\sin\vartheta} \, d\vartheta  \right)\,,
\qquad R\to\infty\,.
\end{eqnarray}
From the inequality
\be\notag
\frac{2\vartheta}{\pi}\le\sin\vartheta\le\vartheta\,,\qquad \vartheta\in\big[0,\tfrac{1}{2}\pi\big]\,,
\ee
it follows that
\be\label{987y4gb7}
\frac{\,1-e^{-\pi^2 R}\,}{2\pi R}\le 
\int_{0}^{\;\frac{\pi}{2}} \!\! e^{-2\pi R\sin\vartheta} \, d\vartheta
\le\frac{1-e^{-2\pi R}}{4 R}\,,
\ee
and since $R$ is large, exponential terms on both sides may be neglected. Thus $I_R=O(1/R)$ at $R\to\infty$.\footnote{Another
way to obtain the same result is to recall that the integral \eqref{987y4gb7} may be evaluated in terms of the modified 
Bessel function $I_n(z)$ of the first kind and the modified Struve function $L_n(z)$. Using the asymptotic expansions of these special functions
we obtain even a more exact result, namely 
\be
\int_{0}^{\frac{\pi}{2}} \!\! e^{-2\pi R\sin\vartheta} \, d\vartheta \,
=\,\frac{\pi}{2}\Big\{I_0(2\pi R) - L_0(2\pi R)\Big\}\,
\sim \,\frac{1}{\,2\pi R\,}\,, \qquad R\to\infty\,,
\ee
i.e.~the integral asymptotically tends to the left bound \eqref{987y4gb7}.}
Inserting this result into \eqref{98ydfn2k}, we obtain 
\begin{eqnarray}\label{8914flehjce}
&& \left|\, \int_{C_R} \!
\frac{\,(a- iz)^{1-s} \,}{\,\ch^2 \! \pi z \,}\,dz \, \right|  \, 
\to0 \qquad\text{as} \quad R\to\infty\,, \, R\in\mathbb{N}\,,\quad
\end{eqnarray}
if $\Re(s)>1$. Hence, making $R\to\infty$, equality \eqref{0923un} becomes
\be\label{674hbu74u}
\int_{-\infty}^{+\infty} \! \frac{\,\big(a - ix\big)^{1-s} \,}{\,\ch^2 \! \pi x \,}\,dx \,=
\ointctrclockwise_{C} \frac{\,\big(a - iz\big)^{1-s} \,}{\,\ch^2 \! \pi z \,}\,dz\,
\ee
where the latter integral is taken around an infinitely large semicircle in the upper half-plane. 
The integrand is not a holomorphic function: it has the poles of the second order at $z=z_n\equiv i\left(n+\frac12\right)$,
$n\in\mathbb{N}_0$, due to the hyperbolic secant, and a branch point 
at $z=-ia$ due to the term in the numerator. If $\Re(a)>0$, the branch point lies outside the integration contour 
and we may use the Cauchy residue theorem:
\begin{eqnarray}\label{089343uy}
&& \ointctrclockwise_{C} \frac{\,\big(a - iz\big)^{1-s} \,}{\,\ch^2 \! \pi z \,}\,dz \, = \,2\pi i 
\!\sum_{n=0}^\infty \res_{z=z_n} \! \frac{\big(a - iz\big)^{1-s}}{\ch^2 \! \pi z}  \, =\\[2mm]
&& \quad \notag
= \,-\frac{2i}{\,\pi\,}
\sum_{n=0}^\infty \left.\frac{\partial }{\partial z} \big(a - iz\big)^{1-s}\right|_{z=i\left(n+\frac12\right)} \!= \\[2mm]
&& \quad \notag
=\, \frac{\,2(s-1)\,}{\,\pi\,} \sum_{n=0}^\infty \big(a + \tfrac12 + n\big)^{-s} \,
=\, \frac{\,2(s-1)\,}{\,\pi\,} \, \zeta\big(s,a+\tfrac12\big)\,.
\end{eqnarray}
Equating \eqref{674hbu74u} with the last result yields
\be
\label{ijh29h3dd2}
\zeta\big(s,a+\tfrac12\big) \,=\,\frac{\,\pi\,}{\,2(s-1)\,} \! \int_{-\infty}^{+\infty}
\! \frac{\,\big(a - ix\big)^{1-s} \,}{\,\ch^2 \! \pi x \,}\,dx \,,\qquad \Re(a) >0 \,.
\ee
Splitting the interval of integration in two parts $(-\infty,0]$ and $[0,+\infty]$ and recalling that 
\be
\big(a + ix\big)^{s} + \big(a - ix\big)^{s} \,=\,2 \big(a^2 + x^2\big)^{\frac{s}{2}}\cos\!\left(\!s\arctg\frac{x}{a}\right)
\ee
the latter expression may also be written as
\begin{eqnarray}\label{834ufn3kl4f}
\zeta\big(s,a+\tfrac12\big) &&  =\,\frac{\,\pi\,}{\,2(s-1)\,} \! \int_{0}^{\infty}
\! \frac{\,\big(a + ix\big)^{1-s} + \big(a - ix\big)^{1-s} \,}{\,\ch^2 \! \pi x \,}\,dx \\[1mm]
&& 
=\,\frac{\,\pi\,}{\,s-1\,} \!\! \int_0^{\infty}\!\!
\! \frac{\, \cos\!\Big[(s-1)\arctg\tfrac{x}{a}\Big] \,}{\,\big(a^2 + x^2\big)^{\frac12(s-1)} \ch^2 \! \pi x \,}\,dx \,,\qquad \Re(a) >0 \,.
\end{eqnarray}
Setting $\,a= v-\tfrac12\,$ in our formulas~for $\zeta\big(s,a+\tfrac12\big)$,
we immediately retrieve our \eqref{9823ydhdhs7}--\eqref{erokvrekov}.
From the principle of analytic continuation it also follows that above integral formulas~are valid for all complex $s\neq1$
and $\Re(a) > 0$ (because of the branch point which should not lie inside the integration contour). 
Note that at $v=1$ our formulas \eqref{9823ydhdhs7}--\eqref{erokvrekov}
reduce to Jensen's formulas for the $\zeta$ function \cite[Eqs.~(88)]{blagouchine2015theorem}.

Now, in order to get the corresponding formulas~for the generalized Stieltjes constant $\gamma_n(v)$ we proceed as follows.
The function $\,(s-1)\zeta\big(s,a+\frac12\big)\,$ is holomorphic on the entire complex $s$--plane,
and hence, may be expanded into a Taylor series. The latter expansion about $s=1$ reads
\begin{eqnarray}\notag
(s-1)\zeta\big(s,a+\tfrac12\big)\,=\,1 
+ \sum_{n=0}^\infty\frac{(-1)^n\gamma_n\big(a+\tfrac12\big)
}{n!}\,(s-1)^{n+1} \,,\quad s\in\mathbb{C}\setminus\!\{1\}.
\end{eqnarray}
But $\,(s-1)\zeta\big(s,a+\frac12\big)\,$ also admits integral representations \eqref{ijh29h3dd2} and \eqref{834ufn3kl4f}. 
Expanding them into the Taylor series in a neighborhood of $s=1$ and equating coefficients in $(s-1)^{n+1}$ 
produces formulas~\eqref{9834df2b3e}--\eqref{9834df2b3e2}. As a particular case of these formulas
we obtain formula (5) from \cite{blagouchine2015theorem} when $v=1$. 
\end{proof}

\begin{corollary}\label{col1}
For complex $v$ and $s$ such that $\Re(v) > \frac12$ and $s\neq1$, 
the Hurwitz zeta function $ \zeta(s,v)$
and the generalized Stieltjes constants $\gamma_n(v)$ admit the representations
\begin{eqnarray}
\zeta(s,v)  &&  =\, \frac{\big(v-\frac12\big)^{1-s}}{\,s-1\,} \, + i \!\int_{0}^{\infty} \!
\frac{\big(v-\tfrac12 - ix\big)^{-s} - \big(v-\tfrac12 + ix\big)^{-s}} {\, e^{2\pi x}+1\,}\, {dx} \,\label{erokvrekov22a}\\[1mm]
&& =\, \frac{\big(v-\frac12\big)^{1-s}}{\,s-1\,} \, - \,2\! \int_{0}^{\infty}\!\!
\! \frac{\, \sin\!\Big[s\arctg\tfrac{2x}{2v-1}\Big] \,}{\,\big(v^2 - v + \tfrac14+ x^2\big)^{\frac{s}{2}} \,}\cdot
\frac{dx}{\, e^{2\pi x}+1 \,} \label{erokvrekov22b}
\end{eqnarray}
and
\be\label{vlkkijvik}
\begin{array}{rl}
\gamma_n(v)\, &\displaystyle  =\, -\frac{\ln^{n+1}\!\big(v-\frac12\big)}{\,n+1\,}\,+  \\[6mm]
&\displaystyle\qquad 
+\, i\!\int_{0}^{\infty} \!
\left\{ \frac{\,\ln^{n}\!\big(v-\tfrac12 - ix\big)} {\, v-\tfrac12 - ix \,} - \frac{\,\ln^{n}\!\big(v-\tfrac12 + ix\big)} 
{\, v-\tfrac12 + ix \,}\,  \right\}
\frac{dx}{\, e^{2\pi x}+1 \,} 
\end{array}
\ee
respectively.
\end{corollary}

\begin{proof}
Let $f(x)$ be such that $f(x)=o(e^{2\pi x})$ as $x\to\infty$. Then, by integration by parts one has
\be\label{87tv987tvct}
\int_0^\infty \!\frac{f(x)}{\, \ch^2\! \pi x \,}\,dx \, =\,\frac{\,f(0)\,}{\pi} +
\frac{\,2\,}{\pi}\!\int_0^\infty \frac{f'(x)}{e^{2\pi x}+1}\,dx\,,
\ee
provided the convergence of both integrals and the existence of $f(0)$. 
Putting $f(x)=\big(a + ix\big)^{1-s} + \big(a - ix\big)^{1-s}$ straightforwardly yields \eqref{erokvrekov22a}.
By virtue of
\be
\big(a + ix\big)^{s} - \big(a - ix\big)^{s} \,=\,2i \big(a^2 + x^2\big)^{\frac{s}{2}}\sin\!\left(\!s\arctg\frac{x}{a}\right),
\ee
we also obtain \eqref{erokvrekov22b}. We remark that at $v=1$ formulas \eqref{erokvrekov22a}--\eqref{erokvrekov22b}
reduce to yet another formula of Jensen for the $\zeta$ function \cite[Eqs.~(88)]{blagouchine2015theorem}
and its differentiated form.
Formula \eqref{vlkkijvik} is obtained analogously from integral \eqref{9834df2b3e2}.
\end{proof}

\begin{remark}\label{y78v98j2}
For real $v> \frac12$, our formulas for the Stieltjes constants may be simplified to
\be\label{oqjhif20hj3e}
\gamma_n(v)\, =\, -\frac{\pi}{\,n+1\,} \Re \left\{\int_{0}^{\infty} 
\!\frac{\,\ln^{n+1}\!\big(v-\frac12\pm ix\big)\,}{\ch^2 \! \pi x}\,dx \right\}
\ee
and to
\begin{eqnarray}
\gamma_n(v)&& =\,-\frac{\ln^{n+1}\!\big(v-\frac12\big)}{\,n+1\,}\,
\pm\, 2\Im \left\{ \int_{0}^{\infty} \!
\frac{\,\ln^{n}\!\big(v-\tfrac12 \pm ix\big)} {\, v-\tfrac12 \pm ix \,} \cdot
\frac{dx}{\, e^{2\pi x}+1 \,}  \right\}\label{vlkkijvik3}
\end{eqnarray}
respectively.
\end{remark}

\begin{remark}\label{kj9j4dk324fr}
Using similar techniques one may obtain many other integral formulas~with kernels decreasing exponentially fast,
for instance:
\begin{eqnarray}
\zeta(s)  &&  
=\, \frac{1}{\,1-2^{s-1}\,}\!\left\{\frac{1}{2} + \frac{1}{\,2i\,}\!\!\int_{-\infty}^{+\infty} \!\!
\big(1 - ix\big)^{-s}\frac{dx }{\,\sh\pi x\,} \right\} = \\[2mm]
&&\qquad\qquad\qquad
=\, \frac{1}{\,1-2^{s-1}\,}\!\left\{\frac{1}{2} \, +\int_{0}^{\infty} \!\!
\frac{\sin\big(s\arctg x \big)}{\,(1+x^2)^{\frac{s}{2}}\sh\pi x\,} \, dx \right\} \,,
\end{eqnarray}

\begin{eqnarray}
\zeta(s,v)  &&  
=\, \pm\frac{i\, \pi^2}{\,(s-1)(s-2)\,}\!\int_{-\infty}^{+\infty} \!\!
\big(v-\tfrac12\pm ix\big)^{2-s} \frac{\sinh\pi x}{\,\ch^3 \! \pi x\,} \, dx \,=  \\[2mm]
&&\quad
=\, \frac{2\, \pi^2}{\,(s-1)(s-2)\,}\!\int_{0}^{\infty} \!\!
\frac{\, \sin\!\Big[(s-2)\arctg\tfrac{2x}{2v-1}\Big] \,}{\,\big(v^2 - v + \tfrac14+ x^2\big)^{\frac{s}{2}-1} \,}\cdot
\frac{\sinh\pi x}{\,\ch^3 \! \pi x\,} \, dx\,,
\end{eqnarray}

\be
\zeta(s,v)  \, =\, \frac{3\pi^3}{\,(s-1)(s-2)(s-3)\,}\!\int_{-\infty}^{+\infty} \!\!
\left\{\!\frac{1}{\,\ch^4 \! \pi x\,} - \frac{2}{\,3\ch^2 \! \pi x\,}\right\} \frac{dx}{\big(v-\frac12\pm ix\big)^{s-3}} \,,
\ee

\be
\gamma_1 \, =\, \frac{\pi^2}{24} - \frac{\gamma^2}{2} -  \frac{\ln^2\! 2}{2} +  \frac{\ln^2\! \pi}{2}  - \ln2\cdot\ln\pi + \int_{0}^{\infty} \!
\frac{\,\arctg x \cdot \ln(1+x^2)\,}{\,\sh\pi x\,} \, dx\,,
\ee

\begin{eqnarray}
&& \gamma_n(v) \, =\, \frac{(-1)^{n+1}}{4\pi}\Big\{n(n-1)\zeta^{(n-2)}(3,v) + 3n\zeta^{(n-1)}(3,v) + 2\zeta^{(n-2)}(3,v)\!\Big\} - \notag\\[2mm]
&& \qquad\qquad\qquad\qquad\qquad
-\frac{3\pi}{\,4(n+1)\,}\!\int_{-\infty}^{+\infty} \!\frac{\,\ln^{n+1}\!\big(v-\frac12\pm ix\big)\,}{\ch^4 \! \pi x}\,dx\,,
\end{eqnarray}
where the latter formulas hold for $\Re (v) > \frac12$ and $n=2,3,4,\ldots$ For the $n=1$ case, one should remove the $n(n-1)\zeta^{(n-2)}(3,v)$
term from the last formula. The previous formulas for $\zeta(s,v)$ and $\gamma_n(v)$ also give rise to corresponding expressions 
for $\Psi(v)$ and $\ln\Gamma(v)$. For example, 
\begin{eqnarray}
&& \Psi(v) \,=\,-\frac{\Psi_2(v)}{4\pi^2} + \frac{3\pi}{4}\!   
\int_{0}^{\infty} \!\frac{\,\ln\big(v^2 - v + \tfrac14+ x^2\big)\,}{\ch^4 \! \pi x}\,dx\,, \notag\\[2mm]
&& \ln\Gamma(v) \,=\, \frac12\ln2\pi + \left(\!v-\frac12\right)\!\Big(\Psi(v)-1\Big) - \pi\!   
\int_{0}^{\infty} \!\frac{\,x\arctg\tfrac{2x}{2v-1}\,}{\ch^2 \! \pi x}\,dx \,, \notag \\[2mm]
&& \ln\Gamma(v) \,=\, \frac12\ln2\pi + \left(\!v-\frac12\right)\!\left(\Psi(v)+\frac{\Psi_2(v)}{4\pi^2}-1\!\right) - \notag\\[1mm]
&& \qquad\qquad\qquad\qquad\qquad\qquad\qquad\quad - \frac{\Psi_1(v)}{4\pi^2} - \frac{\,3\pi\,}{2}\!\!   
\int_{0}^{\infty} \!\frac{\,x\arctg\tfrac{2x}{2v-1}\,}{\ch^4 \! \pi x}\,dx  \,,\notag
\end{eqnarray}
where $\Psi_1(v)$ and $\Psi_2(v)$ are the trigamma and tetragamma functions respectively.\footnote{Some other integral representations
with the kernels decreasing exponentially fast for $\ln\Gamma(v)$ and the polygamma functions
may also be found in \cite{blagouchine2014malmsten} and \cite{blagouchine2015theorem}.  Also, various relationships
between $\ln\Gamma(z)$ and the polygamma functions are given and discussed in \cite{blagouchine2014malmsten}, 
\cite{blagouchine2015theorem} and \cite{blagouchine2018serhasse}.}

It is similarly possible to derive integral representations
for the Lerch transcendent
$\Phi(z,s,v)=\sum_{n=0}^\infty z^n (n+v)^{-s}$, for example
\begin{eqnarray}
&&\Phi(z,s,v) \,  =\, \frac{\,v^{-s}}{2}\, +\frac{\, i\,}{\,2\,} \! \int_{0}^{\infty}
\! \frac{\,(-z)^{ix}\big(v + ix\big)^{-s} -\, (-z)^{-ix}\big(v - ix\big)^{-s} \,}{\,\sh \pi x \,}\,dx \notag\\[1mm]
&& = \, \frac{\,v^{-s}}{2}\, + \int_{0}^{\infty}
\! \frac{\,\cos(x\ln z)\sin\!\Big[s\arctg\tfrac{x}{v}\Big] - \sin(x\ln z)\cos\!\Big[s\arctg\tfrac{x}{v}\Big] \,}
{\,\big(v^2 + x^2\big)^{\frac{s}{2}} \th\pi x \,}\,dx, \notag
\end{eqnarray}
valid for $z>0$, or
\begin{eqnarray}
\zeta(s,v) &&  =\, \frac{\,v^{-s}}{2}\, +\frac{\, i\,}{\,2\,} \! \int_{0}^{\infty}
\! \frac{\,e^{-\pi x}\big(v + ix\big)^{-s} -\, e^{+\pi x}\big(v - ix\big)^{-s} \,}{\,\sh \pi x \,}\,dx \notag\\[1mm]
&&  =\, \frac{\,v^{-s}}{2}\, +\frac{\, i\,}{\,2\,} \! \int_{0}^{\infty}
\! \frac{\,\big(v + ix\big)^{-s} -\,\big(v - ix\big)^{-s} \,}{\,\th \pi x \,}\,dx \notag\\[1mm]
&& = \, \frac{\,v^{-s}}{2}\, + \int_{0}^{\infty}
\! \frac{\,\sin\!\Big[s\arctg\tfrac{x}{v}\Big] \,}
{\,\big(v^2 + x^2\big)^{\frac{s}{2}} \th\pi x \,}\,dx \,,\notag
\end{eqnarray}
whose integrands are not of exponential decay, despite the presence of the hyperbolic cosecant.\footnote{Note that the
second form of these expressions is obtained from the former one by a trivial simplification. Moreover, if we remark 
that $\cth \pi x=1+ 2 \big(e^{2\pi x}-1\big)^{-1}$, we readily notice the relationship between these integrals and the Hermite and Jensen formulas
for the $\zeta$ functions.} 
At the same time, the above formula for $\Phi(z,s,v)$ is suitable for $z>0$, while the same formula for the negative first
argument reads
\begin{eqnarray}
&&\Phi(-z,s,v) \,  =\, \frac{\,v^{-s}}{2}\, +\frac{\, i\,}{\,2\,} \! \int_{0}^{\infty}
\! \frac{\,z^{ix}\big(v + ix\big)^{-s} -\, z^{-ix}\big(v - ix\big)^{-s} \,}{\,\sh \pi x \,}\,dx \notag\\[1mm]
&& = \, \frac{\,v^{-s}}{2}\, + \int_{0}^{\infty}
\! \frac{\,\cos(x\ln z)\sin\!\Big[s\arctg\tfrac{x}{v}\Big] - \sin(x\ln z)\cos\!\Big[s\arctg\tfrac{x}{v}\Big] \,}
{\,\big(v^2 + x^2\big)^{\frac{s}{2}} \sh\pi x \,}\,dx, \notag
\end{eqnarray}
$z>0$, 
and the integrand decreases exponentially fast. 
\end{remark}

\section{Computation of $\gamma_n(v)$ by integration}

\label{sect:computation}
For the computation of $\gamma_n(v)$, we use formulas \eqref{9834df2b3e}--\eqref{9834df2b3e2}, 
\eqref{oqjhif20hj3e}.\footnote{Note that 
formulas \eqref{vlkkijvik}, \eqref{vlkkijvik3} can also provide good computational results.}
For the purpose of brevity throughout this section, we write $a$ for $v - \frac{1}{2}$.
We denote the integrand (with~$a$ and $n$ as implicit parameters) and the half-line integral by
\begin{equation}
\label{eq:f}
f(z) \,\equiv\, \frac{\log^{n+1}(a+iz)}{\cosh^2 \! \pi z }\,, \qquad I_n(a) \,\equiv\, \int_0^{\infty} \!\! f(x) \, dx
\end{equation}
respectively.
After applying \eqref{eq:recurrence} as needed to ensure $\Re(a) > 0$
(or better, $\Re(a) \ge \frac{1}{2}$ to
stay some distance away from the logarithmic branch point
and avoid convergence issues during the
numerical integration to follow),
we may compute
\begin{equation}
\label{eq:gammafromi}
\gamma_n(v) = 
-\frac{\pi}{(n+1)} \cdot
\begin{cases}
2 \Re(I_n(a)), & \Im(a) = 0  \\[2mm]
I_n(a) + \overline{I_n(\overline{a})}, & \Im(a) \ne 0.  
\end{cases}
\end{equation}
where ``$\,\overline{\phantom{m}}\,$'' stands for the complex conjugate.

For a given accuracy goal of $p$ bits, we aim to compute $I_n(a)$ with a relative
error less than $2^{-p}$.
More precisely, we assume use of ball arithmetic~\cite{vdhoeven2009ball},
and we aim to compute an enclosure with relative radius less than $2^{-p}$.
A first important observation is that the computations must
be done with a working precision of about $p + \log_2 n$ bits
for $p$-bit accuracy, due to the sensitivity of the integrand.
In other words, we lose about $\log_2 n$ bits to the exponents
of the floating-point numbers when evaluating exponentials.
Heuristically, a few more guard bits in addition to this will be sufficient
to account for all rounding errors, and the computed ball provides
a final certificate.

A technical point is that we cannot make any a priori statements
about the relative error of $\gamma_n(v)$
since we do not have lower bounds for $|\gamma_n(v)|$.
Cancellation is possible in the final addition (or extraction of the real part) in~\eqref{eq:gammafromi}.
This should roughly correspond
to multiplying by the cosine factor in~\eqref{eq:knessl}; it is reasonable
to set the accuracy goal with respect to the nonoscillatory factor
$B e^{nA} / \sqrt{n}$.

We primarily have in mind ``small'' parameters $v$ (for example $v = 1$)
such that $|v| \ll n$ if $n$ is large.
The algorithm works for any
complex~$v$ where $\gamma_n(v)$ is defined,
but we do not specifically
address optimization for large $|v|$ which therefore
may result in deteriorating efficiency and
less precise output enclosures.

\subsection{Estimation of the tail}

We approximate $I_n(a)$, given by \eqref{eq:f}, by the truncated integral $\int_0^N f(x) \, dx$ for some $N > 0$.
The following theorem provides an upper bound for the tail $T_N$.

\begin{theorem}
Let 
\be
T_N \equiv \int_N^{\infty} \!\frac{\,\log^{n+1}(a+ix)\,}{\cosh^2 \! \pi x } \,  dx 
\ee
and assume $\Re(a) > 0$.
Then, the following bound holds:
\begin{equation}\label{89rh2ioe8}
|T_N| \, < \, 0.934 \,e^{-2\pi N} \big|\log(a+Ni)\big|^{n+1}\,,
\qquad N \ge n + 2 + |\Im(a)|\,.
\end{equation}
\end{theorem}

\begin{proof}
For $x \ge 0$, using $|\log'(z)| = 1/|z|$ and the assumptions on $a$ and $N$ gives
\begin{align*}
\big|\log(a + i (N+x))\big|^{n+1}&  = \big|\log(a + Ni)\big|^{n+1} \left| 1 + \frac{\log(a + i (N+x)) - \log(a + Ni)}{\log(a + Ni)} \right|^{n+1} \\
  &  \le \big|\log(a + Ni)\big|^{n+1} \left(1 + \frac{x}{\big|a + Ni\big| \, \big|\log(a + Ni)\big|} \right)^{n+1} \\
  & \le |\log(a + Ni)|^{n+1} \,\, \exp\!\left(\frac{(n+1) x}{|a + Ni| \, \log|a + Ni|}\right) \\
  & \le |\log(a + Ni)|^{n+1} \, \, \exp(2x).
\end{align*}
Since $\sech^2(x) < 4 e^{-2x}$ on the whole real line, we have
\begin{align*}
|T_N| & \le \int_N^{\infty} 4 e^{-2 \pi x} \left| \log\left(a+ix\right) \right|^{n+1} \,dx \\
& \le 4 e^{-2\pi N} |\log(a + Ni)|^{n+1} \int_0^{\infty} \exp\left(-2 \pi x + 2 x\right) \,dx
\end{align*}
and the last integral equals $\frac{1}{2}(\pi-1)^{-1}.$
\end{proof}

We can select $N$ by starting with $N = n+2+|\Im(a)|$ and repeatedly doubling
$N$ until $|T_N| \le 2^{-p-20}$, say.
This bound does not need to be tight since the integration algorithm,
described later,
discards negligible segments cheaply through bisection.

\begin{remark}
The bound in the previous theorem can be made slightly sharper,
although this does not matter for the algorithm. By using the same line of reasoning as above,
the inequality $\big|\ln(1+z)\big| < \sqrt{|z|}\,$ holding true 
for $\Re(z) \ge 0$ and the fact that the error function is always lesser than 1, one can obtain,
for example,
\be\notag
|T_N| \, < \, 0.637 \Big(1+\tfrac{\lambda}{\sqrt{2}}\, e^{\frac{\lambda^2}{8\pi}} \Big)
e^{-2\pi N} \big|\log(a+Ni)\big|^{n+1}\,, \qquad
N\ge \frac{4(n+1)^2}{\lambda^2}+  |\Im(a)|\,,
\ee
where $\lambda$ is some positive parameter lesser than 2 (the smaller $\lambda$, the better this estimation; 
for $\lambda<0.649$ this estimation outperforms \eqref{89rh2ioe8},
but $N$ must be large with respect to $n^2$). Moreover, for $0<\lambda\ll1$ we may neglect
the term $O(\lambda)$ between the parenthesis, and hence obtain
\be\notag
|T_N| \, \lessapprox \, 0.637\,e^{-2\pi N}\big|\log(a+Ni)\big|^{n+1},
\quad N \gg n^2\,.
\ee
Both these bounds and the value $0.637$, coming from $2/\pi$, are in good agreement with the numerical results.
Note that if \eqref{89rh2ioe8} is suitable for cases in which $N$ is comparable to $n$,
the above estimations are suitable only for cases of large and extra-large $N$ with respect to $n^2$.
\end{remark}

\subsection{Cancellation-avoiding contour}

For small $n$, the integral $\int_0^N f(x) \, dx$ can be computed
directly.
For large $n$, the integrand oscillates on the real line and a higher
working precision must be used due to cancellation.
At least for $v = 1$,
the amount of cancellation can be calculated accurately by
numerically computing the maximum value of $|f(x)|$ on $0 \le x < \infty$
and comparing this magnitude to the asymptotic formula \eqref{eq:knessl}.
For example, we need about 30 extra bits when $n = 10^3$,
1740 bits when $n = 10^6$, and $4 \cdot 10^5$ bits when $n = 10^9$.

For $n$ larger than about $10^3$, we shift the path to eliminate the
cancellation problem. The integrand can be written as
\begin{equation}
\label{eq:fgh1}
f(z) = \exp\left(g(z)\right) h(z),
\end{equation}
where
\begin{equation}
\label{eq:fgh2}
g(z) = (n+1) \log\left(\log\left(a + iz\right)\right) - 2 \pi z, \quad h(z) = (1 + \tanh\pi z)^2.
\end{equation}

Assuming that $n \gg |a|$, the function $\exp(g(z))$ has a single
saddle point in the right half-plane.
The saddle point equation $g'(\omega) = 0$ can be reduced to
\begin{equation}
(n+1) + 2 \pi i \left(a + i \omega \right) \log\left(a + i \omega \right) = 0
\end{equation}
which admits the closed-form solution
\begin{equation}
\omega = i \left( a - \frac{u}{W_0(u)}\right), \quad u = \frac{(n+1) i}{2 \pi}.
\end{equation}
where $W_0(u)$ is the principal branch of the
Lambert $W$ function. Only the principal branch works,
a fact which is not obvious from the symbolic form of the solution
but which can be checked numerically.

We can now integrate along four segments
$$\int_0^N f(x) \, dx \,=\, \int_0^M f(z) \, dz + \int_M^{M+Ci} f(z) \, dz + \int_{M+Ci}^{N+Ci} f(z) \, dz + \int_{N+Ci}^N f(z) \, dz $$
with the choice of vertical offset $C = \operatorname{Im}(\omega)$ to
(approximately) minimize the peak magnitude of $f(t+Ci)$ on $M \le t \le N$.
The left point $M > 0$ just serves to avoid the poles of
the integrand on the imaginary axis and the nearby vertical branch cut
of the logarithm; we can for instance take $M = 10$.

Numerical tests (compare Fig.~\ref{fig:saddle})
confirm that there is virtually no cancellation
with this contour (again, assuming that $|v|$ is not too large).
The path does not exactly pass through the saddle point of $f(z)$, but since
$h(z)$ is exponentially close to a constant, the
perturbation is negligible.
The deviation
between the straight-line path through the saddle point
and the actual steepest descent contour also has negligible impact
on the numerical stability.

\begin{figure}[!t]    
\centering 
\includegraphics[width=0.9\textwidth]{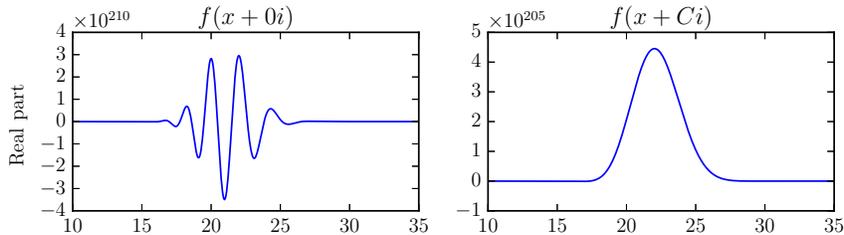} 
\caption{Real part of $f(z)$ for $z = x+0i$ on the real line (left)
and for $z = x+Ci$ passing near the saddle point (right), here with parameters $a = \frac{1}{2}, n = 500$, where integrating along the real line results in about five digits of cancellation.}
\label{fig:saddle} 
\end{figure}

We note that the complex Lambert $W$ function can be computed
with rigorous error bounds~\cite{johansson2017computing}.
However, it is not actually necessary to compute $\omega$ rigorously for
this application since the integration
follows a connected path and ball arithmetic will account for
the actual cancellation; it is sufficient to use a
floating-point approximation for~$\omega$ with heuristic
accuracy of about $\log_2 n$ bits.
For example, an approximation of $\omega$
computed with 53-bit machine arithmetic
is sufficient up to about $n = 10^{15}$.


\subsection{Integration and bounds near the saddle point}

The main task of integrating $f(z)$ along one or four segments in the plane
is not difficult in principle, since $f(z)$ is analytic
(and non-oscillatory) in a neighborhood of each segment.
Constructing a reliable and fast algorithm, in particular
for extremely large $n$, does nevertheless
require some attention to detail.

Gauss-Legendre quadrature is a good option, and was already used
by Ainsworth and Howell~\cite{ainsworth1985}, who,
however, did not prove any error bounds
since ``The integrand is much too complex to use the standard remainder terms''.
To obtain rigorous error bounds and ensure rapid convergence with a manageable
level of manual error analysis, we use the self-validating
Petras algorithm~\cite{petras2002self}
which was recently adapted for arbitrary-precision ball arithmetic
and implemented in the Arb library~\cite{Johansson2018numerical}.

The Petras algorithm combines Gauss-Legendre quadrature
with adaptive bisection.
Given a segment $[\alpha,\beta]$, the algorithm first evaluates the direct
enclosure $(\beta-\alpha) f([\alpha,\beta])$ and uses this if the
error is negligible (which in this application always occurs near the tail ends of the integral
when $p \ll n$).
Otherwise, it bounds
the error of $d$-point quadrature
$$\int_{\alpha}^{\beta} f(z) dz \approx \sum_{k=1}^d w_k f(z_k)$$
in terms of the magnitude
$|f(z)|$ on a Bernstein ellipse $E$ around $[\alpha,\beta]$:
if $f(z)$ is analytic on $E$ with $\max_{z \in E} |f(z)| \le V$,
the error is bounded by $V c / \rho^d$ where $c$ and~$\rho$ only depend
on $E$ and $[\alpha,\beta]$.
If $f(z)$ has poles or branch cuts on $E$
or if the quadrature degree~$d$ determined by this bound
would have to be larger than $O(p)$ to ensure a relative error
smaller than $2^{-p}$, the segment $[\alpha,\beta]$ is bisected
and the same procedure is applied recursively.

The remaining issue is the evaluation of the integrand.
The pointwise evaluations $w_k f(z_k)$ pose no problem:
here we simply use~\eqref{eq:f} directly.
It is slightly more complicated to compute good enclosures for $f(z)$
on wide intervals representing~$z$, which is needed both for
the direct enclosures on subintervals $f([\alpha,\beta])$
and for the bounds on ellipses.\footnote{The complex ball arithmetic in Arb actually uses
rectangles with midpoint-radius real and imaginary parts rather
than complex disks, so ellipses will always be represented by
enclosing rectangles (with up to a factor $\sqrt{2}$ overestimation),
but this detail is immaterial to the principle of the algorithm.}
Bounding the integrand on wide ellipses (or enclosing rectangles) by
evaluating~\eqref{eq:f} or \eqref{eq:fgh1}--\eqref{eq:fgh2}
directly in interval or ball arithmetic
results at best in $n^{1/2+o(1)}$ complexity as $n \to \infty$.\footnote{In fact, the complexity becomes $n^{1+o(1)}$
when using ball arithmetic with a fixed precision for the radii (30 bits in Arb). The $n^{1/2+o(1)}$ estimate holds when the endpoints are tracked accurately.}
The explanation for this phenomenon is
that $f(z)$ is a quotient of two functions $f_1(z) = \log^{n+1}(a+iz)$, $f_2(z) = \cosh^2\!\pi z$ that individually vary
rapidly near the saddle point, i.e.\
\begin{equation}
\label{eq:gheps}
\frac{f_1(z+\varepsilon)}{f_1(z)} \sim \frac{f_2(z+\varepsilon)}{f_2(z)} \sim e^{2 \pi \varepsilon}
\end{equation}
while $f_1 / f_2$ is nearly constant.
Direct evaluation fails to account for this correlation,
which is an example of the dependency problem in interval arithmetic.
Therefore, although $f(z)$ is nearly constant close to the saddle point,
direct upper bounds for $|f(z)|$ are exponentially sensitive
to the width of input intervals, and this forces the integration
algorithm to bisect down to subsegments of width $O(1)$
around the saddle point.
Since the Gaussian peak of the integrand around the saddle point has an effective width of $O(n^{1/2})$,
the integration algorithm has to bisect down to $O(n^{1/2})$ subsegments before converging.

To solve this problem, we compute tighter bounds on wide intervals using the standard trick of
Taylor expanding with respect to a symbolic perturbation $\varepsilon$.

\begin{theorem}
If $z$ is contained in a disk or rectangle $Z$ with midpoint $m$ and radius~$r$,
such that $\Re(Z) \ge 1$, and if $\max_{|u-m| \le r} |g''(u)| \le G$, then
\begin{equation}
\label{eq:fbound}
|f(z)| \; < \;  4.015 \left| \exp\left(g(m)\right) \right| \exp\left(|g'(m)| r + \tfrac{1}{2} G r^2 \right).
\end{equation}
\end{theorem}

\begin{proof}
We use the decomposition \eqref{eq:fgh1}--\eqref{eq:fgh2}.
If $\Re(z) \ge 1$, then $|h(z)| < 4.015$.
Taylor's theorem applied to $\exp(g(z))$ gives
\begin{equation}
\exp\left(g(m + \varepsilon)\right) = \exp\left(g(m)\right) \exp\left(g'(m) \varepsilon + \int_0^{\varepsilon} g''(m+t) (\varepsilon-t) dt\right)
\end{equation}
for all $|\varepsilon| \le r$.
\end{proof}

To implement the bound \eqref{eq:fbound}, we compute $g(m)$ and $g'(m)$ in ball arithmetic
(where $m$ is an exact floating-point number),
using the formula
$$g'(m) = \frac{i (n+1)}{(a + i m) \log(a + i m)} - 2 \pi.$$
The behavior near the saddle point is now captured precisely by the
cancellation in $g'(m)$. At least $\log_2 n$ bits of precision must
be used to evaluate $\exp(g(m))$ (to ensure that the magnitude of the integrand
near the peak is approximated accurately) and also to evaluate $g'(m)$
(to ensure that the remainder after the catastrophic cancellation
is evaluated accurately).
Finally, to compute $G$, we evaluate
$$g''(z) = \frac{(n+1) \left(1 + \frac{1}{\log t}\right)}{t^2 \log t}, \quad t = a + i z$$
directly over the complex ball representing $z$. As a minor optimization,
we can compute lower bounds for $|t|$ and $|\log t|$.
This completes the algorithm.

\subsection{Asymptotic complexity}

If the accuracy goal $p$ is fixed (or grows sufficiently slowly compared to $n$),
then we can argue heuristically that the bit complexity of computing $\gamma_n$ to $p$-bit accuracy
with this algorithm is $\smash{\log^{2+o(1)} n}$.
This estimate accounts for the bisection depth around the
saddle point as well as the extra precision
of $\log_2 n$ bits.
The logarithmic complexity agrees well
with the actual timings (presented in the next section).

We stop short of attempting to prove a formal complexity result,
which would require more detailed calculations
and careful accounting for the accuracy of the enclosures in ball arithmetic
as well as details about the integration algorithm.
We have delegated as much work as possible to a general-purpose
integration algorithm in order to minimize the analysis necessary for a complete implementation.
However, in future work, it would be interesting to pursue such analysis not just for
this specific problem, but more generally for evaluating
classes of parametric integrals using the combination of saddle point analysis
and numerical integration.

If we on the other hand fix $n$ and consider varying $p$, then
the asymptotic bit complexity is of course $p^{2+o(1)}$
since Gaussian quadrature uses $O(p)$ evaluations of the integrand
on a fixed segment and $O(\log p)$ segments are sufficient.

\section{Implementation and benchmark results}

\label{sect:benchmark}

The new integration
algorithm has been implemented in Arb~\cite{Johansson2017arb}.\footnote{\url{http://arblib.org/} -- the new code is available in the 2.14-git version.}
The method \texttt{acb\_dirichlet\_stieltjes} computes $\gamma_n(v)$,
given a complex ball representing $v$, an arbitrary-size integer $n$,
and a precision $p$.
The working precision is set automatically so that the
result will be accurate to about $p$ bits, at least when $v = 1$.
This method selects automatically
between two internal methods:
\begin{itemize}
\item \texttt{acb\_dirichlet\_stieltjes\_integral} uses the new integration algorithm.
\item \texttt{acb\_dirichlet\_stieltjes\_em}
is a wrapper around the existing code for computing
the Hurwitz zeta function using Euler-Maclaurin summation~\cite{Johansson2014hurwitz}.
\end{itemize}

For very small $n$, the integration algorithm is one--three orders
of magnitude slower than Euler-Maclaurin summation,
but the cost of the latter increases rapidly with~$n$.
Integration was found to be faster when $n >  \max(100, p / 2)$,
and this automatic cutoff is used in the code.
We remark that the Euler-Maclaurin code actually computes
$\gamma_0(v), \ldots, \gamma_n(v)$ simultaneously and reads off
the last entry.
At this time, we do not have an implementation of the Euler-Maclaurin
formula optimized for a single $\gamma_n(v)$ value,
which would be significantly faster for $n$ from about $10$ to $10^3$.

Table~\ref{tab:timings} shows the time in seconds to evaluate
the ordinary Stieltjes constants $\gamma_n$
to a target accuracy $p$ of 64 bits (about 18~digits), 333 bits (about 100 digits)
and 3333 bits (just more than 1000 digits)
on an Intel Core i5-4300U CPU running 64-bit Ubuntu Linux.
Here we only show the timing results
for the Arb method \texttt{acb\_dirichlet\_stieltjes\_integral},
omitting use of Euler-Maclaurin summation.
The table also shows timings for
Mathematica 11.0.0 for Microsoft Windows (64-bit) on an Intel Core i9-7900X CPU
for comparison.

As expected, the running time of our algorithm only depends weakly on $n$.
The performance is also reasonable for large $p$.
The timings fluctuate slightly rather than increasing monotonically with $n$,
which appears to be an artifact of the local adaptivity of the
integration algorithm.

Mathematica returns incorrect answers for large $n$ when using machine precision.
At higher precision, the performance is consistent up to about $n = 10^4$,
but the running time then starts to increase rapidly. With $n = 10^5$
and 100-digit or 1000-digit precision, Mathematica
did not finish when left to run overnight.

Mathematica uses Keiper's algorithm according to the documentation~\cite{wolfram},
but unfortunately we do not have details about the implementation.
The timings and failures for large $n$ are seemingly consistent with
use of numerical integration in some form
without the precautions we have taken
against oscillation problems.

\begin{table}[t!]
\renewcommand{\arraystretch}{0.95}
\setlength{\tabcolsep}{1em}
\begin{center}
\caption{Time in seconds to compute $\gamma_n$. The left columns show results for $d$ digits in Mathematica using \texttt{N[StieltjesGamma[n],d]}, or \texttt{N[StieltjesGamma[n]]} when $d=16$ giving machine precision. The smallest results are omitted since the timer in Mathematica does not have sufficient resolution. The (wrong) entries signify that Mathematica returns an incorrect result. The (timeout) entries signify that Mathematica had not completed after several hours. The right columns show results for $p$-bit precision with the new integration algorithm in Arb.\label{tab:timings}}
\begin{tabular}{l | l l l | l l r}
    & \multicolumn{3}{ c }{Mathematica} & \multicolumn{3}{ c }{Arb (integration)} \\
$n$ & $d=16$ & $d=100$ & $d=1000$    & $p=64$ & $p=333$ & \!\!\!\!\!\!$p=3333$ \\ \hline
$1$ &          & &       0.16     & 0.0011 & 0.0089 & 2.7 \\
$10$ &         & 0.016 & 0.39    & 0.0020 & 0.032 & 6.6 \\
$10^2$ & 0.016 & 0.16 & 2.7 & 0.0032 & 0.030 & 3.5 \\
$10^3$ & 0.031 & 0.16 & 3.3 & 0.0064 & 0.10 & 7.5 \\
$10^4$ & (wrong) & 0.41 & 4.5 & 0.0043 & 0.045 & 19.8 \\
$10^5$ & (wrong) & (timeout) & (timeout) & 0.0043 & 0.026 & 27.8 \\
$10^6$ & (wrong) & &     & 0.0066 & 0.026 & 18.1 \\
$10^{10}$ & & &     & 0.0087 & 0.031 & 32.6 \\
$10^{15}$ & & &     & 0.014 & 0.061 & 7.0 \\
$10^{30}$ & & &     & 0.087 & 0.22 & 16.7 \\
$10^{60}$ & & &     & 0.26 & 0.86 & 30.9 \\
$10^{100}$ & & &      & 0.76 & 1.5 & 57.2 \\
\end{tabular}
\end{center}
\end{table}

\begin{table}[t!]
\renewcommand{\arraystretch}{0.95}
\setlength{\tabcolsep}{1em}
\begin{center}
\caption{Time in seconds to compute $\gamma_0, \ldots, \gamma_n$ simultaneously.\label{tab:timings2}}
\begin{tabular}{l | l l l | l l l}
    & \multicolumn{3}{ c }{Arb (Euler-Maclaurin)} & \multicolumn{3}{ c }{Arb (integration)} \\
$n$ &      $p = 64$    &  $p = 333$ & $p = 3333$    & $p = 64$     & $p = 333$ & $p = 3333$ \\ \hline
$1$    & 0.000061 & 0.00026     & 0.012        & 0.012  & 0.12    & 18     \\
$10$   & 0.00035  & 0.0016     & 0.060        & 0.025  & 0.20    & 37     \\
$10^2$ & 0.0047   & 0.11     & 0.39      & 0.28   & 1.8    & 370     \\
$10^3$ & 0.69     & 0.87     & 5.5        & 4.3    & 23    & 4527     \\
$10^4$ & 1207     & 1210     & 1626        & 38     & 267  &      
\end{tabular}
\end{center}
\end{table}

We also mention that Maple is much slower than Mathematica,
taking 0.1 seconds to compute $\gamma_{10}$, a minute to compute $\gamma_{1000}$
and six minutes to compute $\gamma_{2000}$ to 10 digits.

\subsection{Multi-evaluation}

Table~\ref{tab:timings2} compares the performance of Euler-Maclaurin summation
and the integration method in Arb
for computing $\gamma_0(v), \ldots, \gamma_n(v)$ simultaneously.
With the integration algorithm, this means making $n+1$ independent
evaluations, while the Euler-Maclaurin algorithm only has to be executed
once. Despite this, integration still wins for sufficiently
large $n$, unless $p$ also is large.

\subsection{Numerical values}

We show the computed values of a few large Stieltjes constants.
The following significands are correctly rounded to 100 digits (with at most
0.5 ulp error):

\vspace{1.5mm}

$\gamma_{10^{5}} \approx
1.9919273063125410956582272431568589205211659777533113258 \\
75975525936171259272227176914320666190965225 \cdot 10^{83432},$

\vspace{1.5mm}

$\gamma_{10^{10}} \approx
7.5883621237131051948224033799125486921750410324509700470 \\
54093338492423974783927914992046654518550779 \cdot 10^{12397849705},
$

\vspace{1.5mm}

$\gamma_{10^{15}} \approx
1.8441017255847322907032695598351364885675746553315587921 \\
86085948502542608627721779023071573732022221 \cdot 10^{1452992510427658},
$

\vspace{1.5mm}

$
\gamma_{10^{100}} \approx
3.1874314187023992799974164699271166513943099108838469225 \\ 07106265983048934155937559668288022632306095 \cdot 10^e,
$

$e = 2346394292277254080949367838399091160903447689869837 \\ 3852057791115792156640521582344171254175433483694.$

\vspace{1.5mm}

As a sanity check, $\gamma_{10^5}$ agrees with the previous record
Euler-Maclaurin computation~\cite{Johansson2014hurwitz}.
The value of $\gamma_n$ also agrees with
the Knessl-Coffey formula~\eqref{eq:knessl} to about $\log_{10} n$ digits,
in perfect agreement with the error term in this asymptotic approximation
being $O(1/n)$.

For $\gamma_n(v)$ with a nonreal $v$, the computation time
roughly doubles since two integrals are computed.
With $v \ne 1$, we can for instance compute:

\vspace{1.5mm}

$\gamma_{10^{5}}(2+3i) \approx
(1.52933142489317896667092453331813941673604063614322663 \\ 9046917471026123822028695414669890818089958104 \; + \; 7.6266053170235392288 \\
29846454534202735013368165330230700751870950104906000791927387438554979 \\ 23063058i) \cdot 10^{83440},$

\vspace{1.5mm}

$\gamma_{10^{100}}(2+3i) \approx
(0.02447197253567132691871635713584630519276677767177878 \\ 733142765829147799303241971747565188937402242864 
\; + \; 1.328114485458616967 \\ 078662312208319540579816973253179511750642930437359777538176731578318799 
\\ 940692883i) \cdot 10^{e+10}.$

These values similarly agree to $\log_{10} n$ digits with the leading-order truncation
of Paris's generalization~\cite{paris2015asymptotic}
of the Knessl-Coffey formula, providing
both a check on our implementation
and an independent validation of Paris's results.

\section{Discussion}

A few possible optimizations of the integration algorithm are worth pointing out.
The adaptive integration strategy in Arb can probably be improved, which should give a constant factor speedup.
The working precision could also likely be reduced by a preliminary rescaling near the saddle point.

For evaluating a range of $\gamma_n(v)$ simultaneously,
one could perform vector-valued integration and recycle
the evaluations of $\log(a+iz)$ and $\cosh^2\!\pi z$.
It would be interesting to compare this approach to simultaneous evaluation with the Euler-Maclaurin formula.

It would also be interesting to investigate use of
double exponential quadrature instead of Gaussian quadrature.

The computational part of this study was done for two purposes: first, to develop
working code for Stieltjes constants as part of
the collection of rigorous special function routines in the Arb library,
and second, to test the integration algorithm~\cite{petras2002self,Johansson2018numerical}
for a family of integrals involving large parameters.
We do not have a concrete application in mind for the code, but
the Stieltjes constants are potentially useful in
various types of analytic computations involving the Riemann zeta function,
and large-$n$ evaluation can be useful for
testing the accuracy of asymptotic formulas for Stieltjes
constants and related quantities.

The technique of evaluating parametric integrals by integrating numerically
along a steepest descent contour is, of course, well
established in the literature on computational methods for special functions,
but such an algorithm has not previously been published for Stieltjes constants.
The use of rigorous integration techniques in such a setting has also
been explored very little in earlier work.
The most important lesson learned here is that the
heavy lifting can be done by the integration algorithm,
requiring only an elementary pen-and-paper analysis of the integrand.
The same technique should be effective for
rigorously computing many other number sequences and
special functions given by similar integral representations.
On that note, it would be interesting to search for more
integral representations similar to those obtained in Section~\ref{sect:integrals}.
Many such representations with the integrands decreasing exponentially
fast for $\ln\Gamma(z)$ and for the polygamma functions
may be found in \cite{blagouchine2014malmsten}
and \cite{blagouchine2015theorem}.

\section*{Acknowledgements}

We thank Jacques G\'{e}linas for pointing out the previous computations in~\cite{ainsworth1985} and
Vladimir Reshetnikov for helping with some numerical verifications, and are especially greateful to Joseph Oesterl\'e for 
sharing many challenging ideas  on the Stieltjes constants during his stay in St.~Petersburg in June 2017.

\bibliographystyle{plain}

\end{document}